\documentclass[sn-mathphys]{sn-jnl}

\theoremstyle{thmstyleone}%
\newtheorem{theorem}{Theorem}
\newtheorem{prop}[theorem]{Proposition}
\newtheorem{lemma}[theorem]{Lemma}
\newtheorem{cor}[theorem]{Corollary}%

\theoremstyle{thmstyletwo}%
\newtheorem{remark}{Remark}%

\theoremstyle{thmstylethree}%

\newcommand{\sumtwo}{\operatorname*{\sum\sum}}
\newcommand{\lz}{\left(}
\newcommand{\pz}{\right)}
\newcommand{\sumstar}{\sideset{}{^{*}}\sum}
\newcommand{\sumplus}{\sideset{}{^{+}}\sum}

\newcommand{\bfrac}[2]{\lz\frac{#1}{#2}\pz}
\newcommand{\mods}[1]{\: (\mathrm{mod} \: #1)}

\begin{document}
	
	\title[Exceptional characters and non-vanishing]{A note on exceptional characters and non-vanishing of Dirichlet $L$-functions}

	\author*[1]{\fnm{Martin} \sur{\v{C}ech}}\email{martinxcech@gmail.com}
	
	\author[1]{\fnm{Kaisa} \sur{Matom\"aki}}\email{ksmato@utu.fi}

	\affil[1]{\orgdiv{Department of Mathematics and Statistics}, \orgname{University of Turku}, \orgaddress{ \city{Turku}, \postcode{20014}, \country{Finalnd}}}

	\abstract{	We study non-vanishing of Dirichlet $L$-functions at the central point under the unlikely assumption that there exists an exceptional Dirichlet character. In particular we prove that if $\psi$ is a real primitive character modulo $D \in \mathbb{N}$ with $L(1, \psi) \ll (\log D)^{-25-\varepsilon}$, then, for any prime $q \in [D^{300}, D^{O(1)}]$, one has $L(1/2, \chi) \neq 0$ for almost all Dirichlet characters $\chi \mods{q}$.}


	
	\maketitle
	
	\section{Introduction}
	A central problem in analytic number theory is the study of vanishing or non-vanishing of $L$-functions at the central point. Some arithmetic consequences arise for example due to the Birch and Swinnerton-Dyer conjecture, which links the order of the central zero of an elliptic curve $L$-function with its rank (see for example~\cite{Wi}). Another application was provided by Iwaniec and Sarnak~\cite{IS2}, who proved that at least 50\% of $L$-functions in certain families of cusp forms do not vanish at the central point, and showed that any improvement on this proportion would rule out the existence of Landau-Siegel zeros.

	A conjecture attributed to Chowla states that $L(1/2,\chi)\neq 0$ for all Dirichlet characters $\chi$ (see~\cite{Ch} for the conjecture in the case of real Dirichlet characters). In this paper we study the non-vanishing of Dirichlet $L$-functions at the central point under the unlikely assumption that there exists an exceptional Dirichlet character. 
	
	Unconditionally, Balasubramanian and Murty~\cite{BM} were the first to show that for any sufficiently large prime $q$, one has that $L(1/2, \chi) \neq 0$ for a positive proportion of Dirichlet characters $\chi \mods{q}$. This result was significantly improved by Iwaniec and Sarnak~\cite{IS}, who obtained the non-vanishing proportion $1/3-\varepsilon$ (also for non-prime $q$). The best known result for prime moduli is $5/13-\varepsilon$ due to Khan, Mili\'cevi\'c, and Ngo~\cite{KMN}. 
	
	The works of Murty~\cite{Murty} and Bui, Pratt, and Zaharescu~\cite{BPZ} prove the non-vanishing proportion $1/2-o(1)$ conditionally --- Murty under the generalized Riemann hypothsesis, and Bui, Pratt, and Zaharescu under the existence of an exceptional character with modulus of suitable size.

	In this paper we improve on the result of Bui, Pratt, and Zaharescu. In particular we obtain the following corollary, improving their proportion $1/2-o(1)$ to $1-o(1)$.
	
	\begin{cor}\label{cor - main result}
		Let $\varepsilon > 0$ be fixed. Let $D > 1$ be a squarefree fundamental discriminant and let $\psi$ be the associated primitive quadratic character modulo $D$. Assume that \begin{equation} \label{eq:L1upperbound}
			L(1,\psi)\ll \frac{1}{(\log D)^{25+\varepsilon}}.
		\end{equation}
		Then, for any fixed $C>300$ and any prime $q$ such that 
		\begin{equation*}
			D^{300}\leq q\leq D^C,
		\end{equation*} we have
		\begin{equation*}
			\left\lvert\{\chi \mods{q} \colon L(1/2,\chi)\neq 0 \}\right\rvert=(1+o(1))\varphi(q),
		\end{equation*}
		where the rate of convergence of $o(1)$ depends only on $\varepsilon, C$ and the implied constant in~\eqref{eq:L1upperbound}. 
	\end{cor}
	
	\begin{remark}
		\label{rem:improveexc}
		It is feasible that it is possible to loosen the condition~\eqref{eq:L1upperbound} to the condition that $L(1, \psi) = o(1/\log D)$. This would require reworking the arguments in~\cite{BPZ} with a more optimally chosen mollifier than~\eqref{eq:Mdef} below, and being very careful about not losing any logarithmic factors. It might be possible to carry this out by adapting the arguments from~\cite{CI}. In~\cite{CI} Conrey and Iwaniec considered a related problem, showing that if $L(1, \psi) = o(1/\log D)$, then, for any Dirichlet $L$-function, almost all zeros, whose imaginary part is on a suitable range, are simple and lie on the critical line.
	\end{remark}
	
	As in~\cite{BPZ}, we actually get a more quantitative result --- the following holds unconditionally, but is non-trivial only in case an exceptional character exists.
	
	\begin{theorem}\label{thm - main result}
		Let $\varepsilon > 0$ be fixed. Let $D > 1$ be a squarefree fundamental discriminant and let $\psi$ be the associated primitive quadratic character modulo $D$. Let $C>300$ be fixed and let $q$ be a prime such that
		\begin{equation}
			\label{eq:qRange}
			D^{300}\leq q\leq D^C.
		\end{equation}
		Then, for any $\delta > 0$,
		\begin{align}
			\label{eq:maintheorem}
			\frac{1}{\varphi(q)}\sum_{\chi\mods q} \mathbf{1}_{\left\lvert L(1/2,\chi)\right\rvert \geq \frac{\delta^{3/2}}{(\log q)^{9/2}}} = 1 + O \left( \delta^{-2} L(1,\psi) (\log q)^{25+\varepsilon}+ \frac{\delta^{-2}}{(\log q)^{1-\varepsilon}} + \delta \right).
		\end{align}
	\end{theorem}	
	
	Corollary~\ref{cor - main result} immediately follows from applying Theorem~\ref{thm - main result} with $\delta = (\log q)^{-\varepsilon/4}$ and $\varepsilon/3$ in place of $\varepsilon$. In Theorem~\ref{thm - main result} and other statements, the implied constants are allowed to depend on $\varepsilon$ and $C$ (which are said to be fixed), but not on $D$ or $q$.  
	
	\begin{remark} We have not tried to optimize the lower bound we get for $\left\lvert L(1/2, \chi)\right\rvert$ in Theorem~\ref{thm - main result}. By estimating the left hand side of~\eqref{eq:Rest} below more carefully, it would probably be possible to improve the power of $\log q$ in the lower bound. Furthermore, similarly to Remark~\ref{rem:improveexc}, it might be possible to improve on the error term. 
	\end{remark}

	As in many works, we consider only the even primitive characters, the case of odd primitive characters being handled similarly (since $q$ is prime, there is only one non-primitive character $\chi_0$ so its contribution is negligible). We write $\sumplus$ for a sum over primitive even characters modulo $q$, and $\varphi^+(q)$ for the number of such characters.
	
	Our proof is based on the work of Bui, Pratt and Zaharescu \cite{BPZ} and the equidistribution of the product $\varepsilon(\chi)\varepsilon(\psi \chi)$ of root numbers. Here and later, $\varepsilon(\chi)$ denotes the sign of the functional equation of $L(s,\chi)$, which can also be written as a normalized Gauss sum
	\begin{equation}\label{sign as Gauss sum}
		\varepsilon(\chi) :=\frac{\tau(\chi)}{q^{1/2}}=\frac{1}{q^{1/2}}\sum_{\substack{a\mods q \\ (a, q) = 1}}\chi(a)e\bfrac aq.
	\end{equation}
	The following proposition which we will prove in Section~\ref{sec:proofofProp} shows that $\varepsilon(\chi)\varepsilon(\psi \chi)$ is equidistributed on the unit circle when $\chi$ runs over even primitive characters modulo $q$.
	
	\begin{prop}\label{prop - equidistribution}
		Let $D > 1$ be a squarefree fundamental discriminant and let $\psi$ be the associated primitive quadratic character modulo $D$. Let $q$ be a prime such that $q \nmid D$.
		
		For each character $\chi \mods{q}$, let $\theta_{\chi} \in (0, 2\pi]$ be such that $\varepsilon(\chi) \varepsilon(\chi \psi) = e(\theta_\chi)$. Then, for any interval $(\alpha, \beta] \subseteq (0, 2\pi]$, we have
		\begin{equation*}
			\frac{1}{\varphi^+(q)}\,\, \sumplus_{\chi \mods{q}} \mathbf{1}_{\theta_{\chi} \in (\alpha, \beta]} = \beta-\alpha + O(q^{-1/4}),
		\end{equation*}
		where the implied constant is absolute.
	\end{prop}
	
	\subsection*{Acknowledgements}
	The authors thank the anonymous referee for a careful reading of the manuscript and valuable comments.
	The first author was supported by Academy of Finland grant no. 333707 and the second author was supported by Academy of Finland grant no. 285894.

	\section{The work of Bui, Pratt and Zaharescu}
	Let $\varepsilon, D, \psi, C,$ and $q$ be as in Theorem~\ref{thm - main result}. Following~\cite{BPZ}, for any character $\chi \mods{q}$, we write
	$$
	L_\chi(s):=L(s,\chi)L(s,\chi\psi)=\sum_{n\geq 1}\frac{(1*\psi)(n)\chi(n)}{n^s}.
	$$
	Note that Theorem~\ref{thm - main result} is non-trivial only when $\psi$ is an exceptional character modulo $D$ (in the sense that~\eqref{eq:L1upperbound} holds for some $\varepsilon > 0$), and in this case we expect $1*\psi(n)$ to vanish often once $n>D^2$, say (see for example~\cite[formula (2.2)]{BPZ}).
	
	Bui, Pratt and Zaharescu~\cite{BPZ} consider the mollified $L$-functions $L_\chi(1/2)M(\chi)$, where the mollifier is taken as	
	\begin{equation}
		\label{eq:Mdef}
		M(\chi):=\sum_{\substack{n\leq X,\\ D\nmid n}}\frac{(\mu*\mu\psi)(n)\chi(n)}{\sqrt n}
	\end{equation}
	with $X := D^{20}$. In the classical setting, it is crucial to take the mollifier as long as possible to make many of the coefficients of the mollified $L$-function vanish; in the exceptional case, the coefficients $(1*\psi)(n)$ of $L_\chi(1/2)$ are lacunary once $n$ is larger than a small power of $D$, so taking $X=q^{\kappa}$ for a small $\kappa>0$ is sufficient. This explains why we can obtain better non-vanishing results assuming the existence of exceptional characters.
	
	For convenience, we make the same choice of parameters as \cite{BPZ}, so that $X=D^{20}$ and~\eqref{eq:qRange} holds.
	
	Write $Q = q\sqrt{D}/\pi$. Then \cite[formula (4.1)]{BPZ} yields
	\begin{equation}\label{mollified L-function}
		L_\chi(1/2)M(\chi)=V_1\left(\frac{1}{Q}\right)+\varepsilon(\chi)\varepsilon(\chi\psi) V_2\left(\frac{1}{Q}\right)+O\lz\left\lvert B_1(\chi)\right\rvert+\left\lvert B_2(\chi)\right\rvert\pz,
	\end{equation}
	where, for $j = 1, 2$, $V_j(x)$ is a smooth weight as in~\cite[Section 3]{BPZ} and
	\[
	B_j(\chi):=\sum_{\substack{a\leq X,\\ D\nmid a,\\an>1}}\frac{(\mu*\mu\psi)(a)(1*\psi)(n)\chi(an)}{\sqrt{an}}V_j\bfrac{n}{Q}.
	\]
	The formula~\eqref{mollified L-function} is obtained in~\cite{BPZ} using the approximate functional equation (see~\cite[Lemma 3.2]{BPZ}) and isolating the first summand in each term. By~\cite[Lemma 3.4]{BPZ} we have $V_j(1/Q) = 1 + O(Q^{-1/2+\varepsilon})$ for $j=1, 2$ and hence~\eqref{mollified L-function} implies that, for some absolute constant $C_0 \geq 1$,
	\begin{equation}\label{mollified L-function2}
		\left\lvert L_\chi(1/2)M(\chi)-(1 +\varepsilon(\chi)\varepsilon(\chi\psi))\right\rvert \leq C_0\left(\left\lvert B_1(\chi)\right\rvert+\left\lvert B_2(\chi)\right\rvert + Q^{-1/2+\varepsilon}\right).
	\end{equation}
	
	Furthemore, \cite[Proposition 4.1]{BPZ} gives that, for $j = 1, 2$, any $\varepsilon > 0$, and any prime $q$ in the range~\eqref{eq:qRange},
	\begin{equation}\label{second moment of B_i}
		\sumplus_{\chi\mods q}\left\lvert B_j(\chi)\right\rvert ^2\ll L(1,\psi) q(\log q)^{25+\varepsilon}+\frac{q}{(\log q)^{1-\varepsilon}}. 
	\end{equation}
	
	The strategy of Bui, Pratt and Zaharescu~\cite{BPZ} is to proceed with the usual method of applying the Cauchy-Schwarz inequality to obtain that
	\begin{equation}
		\label{eq:BPZC-S}
		\sumplus_{\substack{\chi \mods{q} \\ L(1/2, \chi) \neq 0}} 1 \geq \sumplus_{\substack{\chi \mods{q} \\ L_\chi(1/2) M(\chi) \neq 0}} 1 \geq \frac{\left\lvert\sideset{}{^{+}_{\chi \mods{q}}}\sum L_\chi(1/2) M(\chi)\right\rvert^2}{\sideset{}{^{+}_{\chi \mods{q}}}\sum \left\lvert L_\chi(1/2) M(\chi)\right\rvert ^2}.
	\end{equation}
	Bui, Pratt and Zaharescu then use~\eqref{mollified L-function} and~\eqref{second moment of B_i} to compute the first and second moments of $L_\chi(1/2)M(\chi)$ --- this gives that (assuming that $\psi$ is an exceptional character), on the right hand side of~\eqref{eq:BPZC-S}, the numerator equals $(1+o(1))\varphi^+(q)^2$ whereas the denominator equals $(2+o(1))\varphi^+(q)$. This yields the non-vanishing proportion $1/2+o(1)$.
	
	Note that the application of the Cauchy-Schwarz inequality in~\eqref{eq:BPZC-S} is costly, because by~\eqref{mollified L-function2} the mollified $L$-functions still oscillate --- our strategy is to dispose of the use of the Cauchy-Schwarz inequality and instead exploit the fact that \eqref{mollified L-function2} holds for individual L-functions. Actually, from~\eqref{mollified L-function2}, the equidistribution of the signs $\varepsilon(\chi)\varepsilon(\chi\psi)$ (Proposition~\ref{prop - equidistribution}) and~\eqref{second moment of B_i}, one can see that $L_\chi(1/2)M(\chi)$ for even $\chi \mods{q}$ are equidistributed in the circle $\left\lvert z-1\right\rvert = 1$, which directly implies the non-vanishing of $L_\chi(1/2)M(\chi)$ for almost all characters $\chi\mods q$.
	
	The variation of the root number has been utilized also in earlier (unconditional) results that involved a two-piece mollifier (see e.g. \cite{MV} and \cite{KMN}) --- in these works one used the Cauchy-Schwarz inequality, but optimized its application.
	
	\section{Proof of Theorem~\ref{thm - main result} assuming Proposition~\ref{prop - equidistribution}} 
	In this section we prove Theorem~\ref{thm - main result} assuming Proposition~\ref{prop - equidistribution}. Let $\varepsilon, D, \psi, C, q,$ and $\delta$ be as in Theorem~\ref{thm - main result} and let $C_0$ be as in~\eqref{mollified L-function2}. Now~\eqref{second moment of B_i} implies that
	\[
	\sumplus_{\chi\mods q} \mathbf{1}_{\left\lvert B_j(\chi)\right\rvert \geq \delta/(4C_0)} \ll \delta^{-2} \left(L(1,\psi) q(\log q)^{25+\varepsilon/2}+\frac{q}{(\log q)^{1-\varepsilon/2}}\right).
	\]
	Furthermore, by Proposition~\ref{prop - equidistribution} we know that $\left\lvert 1+\varepsilon(\chi) \varepsilon(\chi \psi)\right\rvert  \geq 2\delta$ for all $\chi \mods{q}$ apart from an exceptional set consisting of $\ll \delta \varphi(q) + q^{3/4}$ characters.
	
	Combining~\eqref{mollified L-function2} with the triangle inequality and these observations we obtain that, apart from an exceptional set of size
	\begin{equation}
		\label{eq:ExceptionalSet}
		\ll \delta^{-2} L(1,\psi) q(\log q)^{25+\varepsilon}+\delta^{-2} \frac{q}{(\log q)^{1-\varepsilon}} + \delta \varphi(q),
	\end{equation}
	we have
	\begin{equation}
		\label{eq:LMlower}
		\left\lvert L_\chi(1/2) M(\chi)\right\rvert \geq 2\delta- C_0\left(\frac{\delta}{4C_0} + \frac{\delta}{4C_0} + Q^{-1/2+\varepsilon}\right) > \delta. 
	\end{equation}
	This already yields~\eqref{eq:maintheorem} with $\mathbf{1}_{\left\lvert L(1/2, \chi)\right\rvert \geq \delta^{3/2}/(\log q)^{9/2}}$ replaced by $\mathbf{1}_{\left\lvert L(1/2, \chi)\right\rvert \neq 0}$ and is thus sufficient for obtaining Corollary~\ref{cor - main result}. We next proceed to showing a lower bound for $\left\lvert L(1/2, \chi)\right\rvert $ outside an acceptable exceptional set.	
	
	Recall that 
	\[
	L_\chi(1/2) M(\chi) = L(1/2, \chi) L(1/2, \chi\psi) M(\chi).
	\] 
	Since~\eqref{eq:LMlower} holds apart from an exceptional set of size~\eqref{eq:ExceptionalSet}, Theorem~\ref{thm - main result} follows if we can establish that there are at most $O(\delta \varphi(q))$ characters $\chi \mods{q}$ for which $\left\lvert L(1/2, \chi\psi) M(\chi)\right\rvert \geq \delta^{-1/2} (\log q)^{9/2}$. This follows if
	\begin{equation}
		\label{eq:quantbound}
		\sum_{\substack{\chi \mods{q} \\ \chi \neq \chi_0}} \left\lvert L(1/2, \chi\psi) M(\chi)\right\rvert^2 \ll \varphi(q) (\log q)^9,
	\end{equation}
	and so it suffices to establish~\eqref{eq:quantbound}. 
	
	By the approximate functional equation we have for any primitive character $\chi$ (see e.g.~\cite[formula (2.2)]{IS})
	\begin{equation}
		\label{eq:AFELchipsi}
		L(1/2, \chi \psi) = \sum_{n = 1}^\infty \frac{\chi(n)\psi(n) + \varepsilon(\chi \psi) \overline{\chi }(n)\overline{\psi}(n)}{\sqrt{n}} W\lz n\sqrt{\pi/(qD)}\pz,
	\end{equation}
	where $W$ (denoted by $V$ in~\cite{IS}) is such that $W(y) = 1+ O(y^{10})$ and $W(y) \ll y^{-10}$. Using these bounds we see that 
	\[
	L(1/2, \chi \psi) = \sum_{n \leq (qD)^{\frac34}} \frac{\chi(n)\psi(n) + \varepsilon(\chi \psi) \overline{\chi }(n)\overline{\psi}(n)}{\sqrt{n}} W\lz n\sqrt{\pi/(qD)}\pz + O\left(\frac{1}{qD}\right).
	\]
	Using also the definition of $M(\chi)$ (see~\eqref{eq:Mdef}) and the orthogonality of characters (adding back $\chi=\chi_0$), noting that $X(qD)^{3/4} \leq q$, we obtain
	\begin{align}
		\begin{aligned}
			\label{eq:Rest}
			&\sum_{\substack{\chi \mods{q} \\ \chi \neq \chi_0}} \left\lvert L(1/2, \chi\psi) M(\chi)\right\rvert^2 \\
			&\ll \sum_{\substack{\chi \mods{q}}} \vert M(\chi)\vert^2 \left\vert \sum_{n \leq (qD)^{\frac34}} \frac{\chi(n)\psi(n) + \varepsilon(\chi \psi) \overline{\chi }(n)\overline{\psi}(n)}{\sqrt{n}} W\lz n\sqrt{\pi/(qD)}\pz\right\vert^2 \\
			& \qquad + \frac{1}{q^2D^2} \sum_{\substack{\chi \mods{q}}} \vert M(\chi) \vert^2 \\
			&\ll \varphi(q) \sumtwo_{\substack{k_1, k_2 \leq (qD)^{3/4} \\ \ell_1, \ell_2 \leq X \\ k_1 \ell_1 = k_2 \ell_2}} \frac{\left\lvert(\mu \ast \mu\psi)(\ell_1)\right\rvert \left\lvert(\mu \ast \mu\psi)(\ell_2)\right\rvert}{\sqrt{k_1\ell_1 k_2 \ell_2}} + \frac{\varphi(q)}{q^2D^2} \sum_{\substack{n \leq X \\ D \nmid n}} \frac{\vert (\mu \ast \mu \psi)(n)\vert^2}{n} \\
			&\ll \varphi(q) \sum_{n \leq X (qD)^{3/4}} \frac{d_3(n)^2}{n} + \frac{X^{\varepsilon}}{qD^2} \ll  \varphi(q) \prod_{p \leq X (qD)^{3/4}} \left(1+\frac{3^2}{p}\right).
		\end{aligned}	
	\end{align}
	Now~\eqref{eq:quantbound} follows from Mertens' theorem, and so the proof of Theorem~\ref{thm - main result} is completed.
	
	The remaining task is to prove Proposition \ref{prop - equidistribution} which will be done in the following section.
	\section{Proof of Proposition \ref{prop - equidistribution}}
	\label{sec:proofofProp}
	In this section we prove Proposition~\ref{prop - equidistribution}. By the Erd\H{o}s-Tur\'an inequality (see e.g.~\cite[Corollary 1.1]{Montgomery} with $K = \lfloor q^{1/4} \rfloor$),
	\[
	\begin{aligned}
		\Bigg\vert\frac{1}{\varphi^+(q)}\,\, &\sumplus_{\chi \mods{q}} \mathbf{1}_{\theta_{\chi} \in (\alpha, \beta]} - (\beta-\alpha)\Bigg\vert\leq \frac{1}{q^{1/4}} + \frac{3}{\varphi^+(q)}\sum_{1 \leq k \leq q^{1/4}} \frac{1}{k} \left\lvert\quad \sumplus_{\chi \mods{q}} e(k \theta_{\chi})\right\rvert.
	\end{aligned}
	\]
	Since $e(k \theta_{\chi}) = (\varepsilon(\chi) \varepsilon(\chi \psi))^k$, Proposition \ref{prop - equidistribution} follows immediately from the following lemma.
	\begin{lemma}
		Let $k \in \mathbb{N}$.  Let $D > 1$ be a square-free fundamental discriminant and let $\psi$ be the associated primitive quadratic character modulo $D$. Let $q$ be a prime with $q \nmid D$. Then
		\[
		\left\lvert\ \, \sumplus_{\chi \mods{q}} \left(\varepsilon(\chi) \varepsilon(\chi \psi)\right)^k\right\rvert\leq \frac{1}{q^k} + 2k\cdot\frac{\varphi(q)}{q^{1/2}}.
		\]
	\end{lemma}
	\begin{proof}
		Our argument generalizes an argument in~\cite[Section 4]{BPZ} where the special case $k=1$ was established (see~\cite[formula (4.6)]{BPZ}).	Since $(q, D) = 1$, one gets (as pointed out in~\cite[page 603]{BPZ}) from~\eqref{sign as Gauss sum} and the Chinese remainder theorem 
		\begin{align*}
			\varepsilon(\chi \psi) &= \frac{1}{(Dq)^{1/2}} \sum_{\substack{a\mods{Dq} \\ (a, Dq) = 1}}\chi(a)\psi(a) e\bfrac{a}{Dq} \\
			&= \frac{1}{(Dq)^{1/2}} \sum_{\substack{b \mods D \\ (b, D) = 1}} \sum_{\substack{c \mods q \\ (c, q) = 1}} \chi(bq+cD)\psi(bq+cD) e\bfrac{bq+cD}{Dq} \\
			&= \chi(D) \psi(q) \varepsilon(\psi)\varepsilon(\chi).
		\end{align*}
		Hence
		\[
		\varepsilon(\chi) \varepsilon(\chi \psi) = \chi(D)\psi(q)\varepsilon(\psi)\varepsilon(\chi)^2
		\]
		and, using also \eqref{sign as Gauss sum},
		\begin{align}
			\label{eq:root1k}
			\begin{aligned}
				&\sumplus_{\substack{\chi \mods{q}}} \left(\varepsilon(\chi) \varepsilon(\chi \psi)\right)^k \\
				&= \frac{\psi(q)^k\varepsilon(\psi)^k}{q^k} \sum_{\substack{a_1, \dotsc, a_{2k}\mods q \\ (a_j, q) = 1}} e\left(\frac{a_1 + \dotsb + a_{2k}}{q}\right) \sumplus_{\chi \mods{q}} \chi(D^k a_1 \dotsm a_{2k}).
			\end{aligned}
		\end{align}
		By orthogonality of characters we have, for any prime $q$ and integers $m, n$ such that $(mn, q) = 1$,
		\begin{align*}
			\sumplus_{\chi\mods q}\chi(m)\overline\chi(n)&=\frac{1}{2} \sumstar_{\chi\mods q}(1+\chi(-1))\chi(m)\overline\chi(n) \\
			&= \frac{1}{2} \sum_{\chi\mods q}(1+\chi(-1))\chi(m)\overline\chi(n) - 1 =\mathbf{1}_{q\mid(m\pm n)}\frac{\varphi(q)}{2}-1.
		\end{align*}
		
		Applying this to~\eqref{eq:root1k}, we obtain
		\begin{align}
			\begin{aligned}
				\label{eq:rootOut}
				\sumplus_{\substack{\chi \mods{q}}} \left(\varepsilon(\chi) \varepsilon(\chi \psi)\right)^k &= \frac{\psi(q)^k\varepsilon(\psi)^k}{2q^k} \varphi(q) \sum_{\substack{a_1, \dotsc, a_{2k} \mods q\\ D^k a_1 \dotsm a_{2k} \equiv \pm 1 \mods{q}}} e\left(\frac{a_1 + \dotsb + a_{2k}}{q}\right) \\
				&\qquad  - \frac{\psi(q)^k\varepsilon(\psi)^k}{q^k} \sum_{\substack{a_1, \dotsc, a_{2k} \mods q\\ (a_j, q) = 1}} e\left(\frac{a_1 + \dotsb + a_{2k}}{q}\right).
			\end{aligned}
		\end{align}
		The second term on the right hand side of~\eqref{eq:rootOut} equals
		\[
		-\frac{\psi(q)^k\varepsilon(\psi)^k}{q^k} \left(\sum_{a=1}^{q-1} e\left(\frac{a}{q}\right)\right)^{2k} = -\frac{\psi(q)^k\varepsilon(\psi)^k}{q^k},
		\]
		and thus has absolute value at most $1/q^k$.
		
		On the other hand, the first term on the right hand side of~\eqref{eq:rootOut} equals
		\begin{align}
			\label{eq:firstterm}
			&  \frac{\psi(q)^k\varepsilon(\psi)^k}{2q^k} \varphi(q) \sum_{\ell=0}^1 \sum_{\substack{a_1, \dotsc, a_{2k-1}\mods q \\ (a_j, q) = 1}} e\left(\frac{a_1 + \dotsb + a_{2k-1} + (-1)^\ell \overline{D^k a_1 \dotsm a_{2k-1}}}{q}\right).
		\end{align}
		Here we have a $2k-1$-dimensional Kloosterman sum and by a bound of Smith \cite[Theorem 6]{Smith}, the absolute value of~\eqref{eq:firstterm} is
		\[
		\leq \frac{\varphi(q)}{2q^k} \cdot 2 \cdot q^{(2k-1)/2}d_{2k}(q) = \frac{\varphi(q)}{q^{1/2}} d_{2k}(q).
		\]
		Now the claim follows since $q$ is a prime, so $d_{2k}(q) = 2k$.
	\end{proof}
	
	\section*{Statements and Declarations}
	On behalf of all authors, the corresponding author states that there is no conflict of interest.


\begin{thebibliography}{10}
		
		\bibitem{BM}
		R.~Balasubramanian and V.~Kumar Murty.
		\newblock Zeros of {D}irichlet {$L$}-functions.
		\newblock {\em Ann. Sci. \'{E}cole Norm. Sup. (4)}, 25(5):567--615, 1992.
		
		\bibitem{BPZ}
		Hung~M. Bui, Kyle Pratt, and Alexandru Zaharescu.
		\newblock Exceptional characters and nonvanishing of {D}irichlet
		{$L$}-functions.
		\newblock {\em Math. Ann.}, 380(1-2):593--642, 2021.
		
		\bibitem{Ch}
		S.~Chowla.
		\newblock {\em The {R}iemann hypothesis and {H}ilbert's tenth problem}.
		\newblock Mathematics and its Applications, Vol. 4. Gordon and Breach Science
		Publishers, New York-London-Paris, 1965.
		
		\bibitem{CI}
		J.~B. Conrey and H.~Iwaniec.
		\newblock Critical zeros of lacunary {$L$}-functions.
		\newblock {\em Acta Arith.}, 195(3):217--268, 2020.
		
		\bibitem{IS}
		H.~Iwaniec and P.~Sarnak.
		\newblock Dirichlet {$L$}-functions at the central point.
		\newblock In {\em Number theory in progress, {V}ol. 2
			({Z}akopane-{K}o\'{s}cielisko, 1997)}, pages 941--952. de Gruyter, Berlin,
		1999.
		
		\bibitem{IK}
		Henryk Iwaniec and Emmanuel Kowalski.
		\newblock {\em Analytic number theory}, volume~53 of {\em American Mathematical
			Society Colloquium Publications}.
		\newblock American Mathematical Society, Providence, RI, 2004.
		
		\bibitem{IS2}
		Henryk Iwaniec and Peter Sarnak.
		\newblock The non-vanishing of central values of automorphic {$L$}-functions
		and {L}andau-{S}iegel zeros.
		\newblock {\em Israel J. Math.}, 120(part A):155--177, 2000.
		
		\bibitem{KMN}
		Rizwanur Khan, Djordje Mili\'{c}evi\'{c}, and Hieu~T. Ngo.
		\newblock Nonvanishing of {D}irichlet {$L$}-functions, {II}.
		\newblock {\em Math. Z.}, 300(2):1603--1613, 2022.
		
		\bibitem{MV}
		Philippe Michel and Jeffrey VanderKam.
		\newblock Non-Vanishing of High Derivatives of {D}irichlet
		{$L$}-Functions at the Central Point, {II}.
		\newblock {\em J. Number Theory}, 81, 130--148 (2000).
		
		\bibitem{Montgomery}
		Hugh~L. Montgomery.
		\newblock {\em Ten lectures on the interface between analytic number theory and
			harmonic analysis}, volume~84 of {\em CBMS Regional Conference Series in
			Mathematics}.
		\newblock Published for the Conference Board of the Mathematical Sciences,
		Washington, DC; by the American Mathematical Society, Providence, RI, 1994.
		
		\bibitem{Murty}
		M.~Ram Murty.
		\newblock On simple zeros of certain {$L$}-series.
		\newblock In {\em Number theory ({B}anff, {AB}, 1988)}, pages 427--439. de
		Gruyter, Berlin, 1990.
		
		\bibitem{Smith}
		Robert~A. Smith.
		\newblock On {$n$}-dimensional {K}loosterman sums.
		\newblock {\em J. Number Theory}, 11(3, S. Chowla Anniversary Issue):324--343,
		1979.
		
		\bibitem{Wi}
		Andrew Wiles.
		\newblock The {B}irch and {S}winnerton-{D}yer conjecture.
		\newblock In {\em The millennium prize problems}, pages 31--41. Clay Math.
		Inst., Cambridge, MA, 2006.
		
	\end{thebibliography}
\end{document}